\documentclass[12pt]{amsproc}

\author{Emerson de Melo}
\address{Department of Mathematics, University of Bras\'ilia, Bras\'ilia-DF 70910-900, Brazil}
\email{emerson@mat.unb.br}

\author{Aline de Souza Lima}

\address{Department of Mathematics and Statistics, Federal University of Goi\'as, Goi\^ania-GO, 74001-970, Brazil}

\email{alinelima@ufg.br}

\author{Pavel Shumyatsky}
\address{Department of Mathematics, University of Bras\'ilia,
Bras\'ilia-DF, 70910-900, Brazil }
\email{pavel@unb.br}

\thanks{The first author was supported by FEMAT; The third author was supported by FAPDF and CNPq-Brazil}

\title{Nilpotent residual of fixed points}
\date{2017}

\newtheorem{theorem}{\sc Theorem}[section]
\newtheorem{lemma}[theorem]{\sc Lemma}

\begin{document}

%\maketitle

\begin{abstract}
Let $q$ be a prime and $A$ a finite $q$-group of exponent $q$ acting by automorphisms on a finite $q'$-group $G$. Assume that $A$ has order at least $q^3$. We show that if $\gamma_{\infty} (C_{G}(a))$ has order at most $m$ for any $a \in A^{\#}$, then the order of $\gamma_{\infty} (G)$ is bounded solely in terms of $m$ and $q$. If $\gamma_{\infty} (C_{G}(a))$ has rank at most $r$ for any $a \in A^{\#}$, then the rank of $\gamma_{\infty} (G)$ is bounded solely in terms of $r$ and $q$.
\end{abstract}

\maketitle

\section{Introduction}

Suppose that a finite group $A$ acts by automorphisms on a finite group $G$.  The action is coprime if the groups $A$ and $G$ have coprime orders. We denote by $C_G(A)$ the set $$C_G(A)=\{g\in G\ |\ g^a=g \ \textrm{for all} \ a\in A\},$$ the centralizer of $A$ in $G$ (the fixed-point subgroup). In what follows we denote by $A^\#$ the set of nontrivial elements of $A$. It has been known that centralizers of coprime automorphisms have strong influence on the structure of $G$.

Ward showed that if $A$ is an elementary abelian $q$-group of rank at least 3 and if $C_G(a)$ is nilpotent for any $a\in A^\#$, then the group $G$ is nilpotent \cite{War}. Later the third author showed that if, under these hypotheses, $C_G(a)$ is nilpotent of class at most $c$ for any $a\in A^\#$, then the group $G$ is nilpotent with $(c,q)$-bounded nilpotency class \cite{Sh}. Throughout the paper we use the expression ``$(a,b,\dots )$-bounded'' to abbreviate ``bounded from above in terms of  $a,b,\dots$ only''. In the recent article \cite{Eme1} the above result was extended to the case where $A$ is not necessarily abelian. Namely, it was shown that if $A$ is a finite group of prime exponent $q$ and order at least $q^3$ acting on a finite $q'$-group $G$ in such a manner that $C_G(a)$ is nilpotent of class at most $c$ for any $a\in A^{\#}$, then $G$ is nilpotent with class bounded solely in terms of $c$ and $q$. Many other results illustrating the influence of centralizers of automorphisms on the structure of $G$ can be found in \cite{khukhro}.

In the present article we study finite groups $G$ acted on by a (possibly non-abelian) group $A$ of prime exponent $q$ and order at least $q^3$ such that $C_G(a)$ has ``small" nilpotent residual for every $a\in A^{\#}$. Recall that the nilpotent residual $\gamma_\infty(K)$ of a group $K$ is the last term of the lower central series of $K$. It can also be defined as the intersection of all normal subgroups of $K$ whose quotients are nilpotent. The order of a finite group $K$ is denoted by $|K|$. The rank of (a finite group) $K$ is the minimal number $r$, denoted by ${\bf r}(K)$, such that every subgroup of $K$ can be generated by at most $r$ elements. Guralnick \cite{G1} and, independently, Lucchini \cite{Lu} proved that ${\bf r}(K) \le 1 + \max_p \{{\bf r}(P) \ | \ P \ \hbox{a Sylow $p$-subgroup of} \ K\}$.

 We obtain the following results.

\begin{theorem}\label{main1}
Let $q$ be a prime and $A$ a finite $q$-group of exponent $q$ acting by automorphisms on a finite $q'$-group $G$. Assume that $A$ has order at least $q^3$ and $|\gamma_{\infty} (C_{G}(a))| \leq m$ for any $a\in A^{\#}$. Then $\gamma_{\infty}(G)$ has $(m,q)$-bounded order.   
\end{theorem}

\begin{theorem}\label{main2}
Let $q$ be a prime and $A$ a finite $q$-group of exponent $q$ acting by automorphisms on a finite $q'$-group $G$. Assume that $A$ has order at least $q^3$ and ${\bf r}(\gamma_{\infty} (C_{G}(a)))\leq r$ for any $a \in A^{\#}$. Then ${\bf r}(\gamma_{\infty} (G))$ is $(r,q)$-bounded.   
\end{theorem}

Unsurprisingly, these results depend on the classification of finite simple groups and it seems unlikely that one could find a classification-free proof.

\section{Preliminaries}
If $A$ is a group of automorphisms of a group $G$, the subgroup generated by elements of the form $g^{-1}g^\alpha$ with $g\in G$ and $\alpha\in A$ is denoted by $[G,A]$. It is well-known that the subgroup $[G,A]$ is an $A$-invariant normal subgroup in $G$. Our first lemma is a collection of well-known facts on coprime actions (see for example \cite{GO}). Throughout the paper we will use it without explicit references.
\begin{lemma}
Let  $A$ be a group of automorphisms of a finite group $G$ such that $(| G |, | A |) = 1$. Then
\begin{enumerate}
\item[i)] $G = C_{G}(A)[G,A]$.
\item[ii)] $[G,A,A]=[G,A]$.
\item[iii)] $A$ leaves invariant some Sylow $p$-subgroup of $G$ for each prime $p\in\pi(G)$.
\item[iv)] $C_{G/N}(A)=C_G(A)N/N$ for any $A$-invariant normal subgroup $N$ of $G$.
\item[v)] If $G$ is nilpotent and $A$ is a noncyclic abelian group, then $G=\prod_{a\in A^{\#}}C_{G}(a)$.
\end{enumerate}
\end{lemma}

We will require the following well-known fact.
 
\begin{lemma}\label{C1} Let $N,H_1,\dots,H_l$ be subgroups of a group $G$ with $N$ being normal. If $K = \langle H_1,\dots,H_l\rangle$, then $[N, K] = [N, H_1]\dots[N,H_l]$.
\end{lemma}

Throughout the rest of this section we will assume the following hypothesis. As usual, $Z(K)$ stands for the center of a group $K$.

Let $q$ be a prime and $A$ a group of exponent $q$ and order $q^3$. Denote by $B$ a subgroup of order $q$ of $Z(A)$. Let $A$ act on a finite $q'$-group $G=PH$, where $P$ and $H$ are $A$-invariant subgroups such that $P$ is a normal $p$-subgroup for a prime $p$ and $H$ is a nilpotent $p'$-subgroup. It is clear that $A$ has precisely $q+1$ subgroups of order $q^2$ containing $B$.

\begin{lemma}\label{R1}
Let $A_1,\dots,A_{q+1}$ be the subgroups of order $q^2$ of $A$ containing $B$. Then $C_P(B)=\prod C_{P}(A_{i})$ and $C_H(B)=\prod C_{H}(A_{i})$.
\end{lemma}
\begin{proof}
Denote by $\overline{A}$ the quotient-group $A/B$. Since $\overline{A}$ is not cyclic and both centralizers $C_P(B)$ and $C_H(B)$ are $A$-invariants, it follows that $C_P(B)=\prod C_{C_P(B)}(\overline{a})$ and  $C_H(B)=\prod C_{C_H(B)}(\overline{a})$ where  $\overline{a}\in \overline{A}^{\#}$. An alternative way of expressing this is to write that $C_P(B)=\prod C_{P}(A_{i})$ and $C_H(B)=\prod C_{H}(A_{i})$. 
\end{proof}

\begin{lemma}\label{comm}
Suppose that $P$ is abelian. Then $[P,C_H(B)]$ is contained in $\prod[C_P(a),C_H(a)]$, where the product is taken over all $a\in A^{\#}$.
\end{lemma}
\begin{proof}
In view of Lemma \ref{R1} we have $C_H(B)=\prod C_{H}(A_{i})$ and $P=\prod_{a\in A_i^{\#}} C_P(a)$ for each $i$. By Lemma \ref{C1}, $[P,C_H(B)]=\prod [P,C_H(A_i)]$. Since $P$ is abelian, for each $i$ we have $[P,C_H(A_i)]=\prod [C_P(a),C_H(A_i)]$ where the product is taken over all $a\in A_i^{\#}$. In particular, $[P,C_H(A_i)]=\prod [C_P(a),C_H(a)]$.
\end{proof}

\begin{lemma}\label{commT}
If $[C_P(a),C_H(a)] =1$ for any $a\in A^{\#}$, then $[P,H]=1$.
\end{lemma}
\begin{proof}

First we assume that $P$ is abelian. By Lemma \ref{comm}, $[P,C_H(B)]=1$. Let us prove that $[C_P(B),H]=1$. Using the notation of Lemma \ref{R1} we have that $C_P(B)=\prod C_P(A_i)$ and $H=\prod_{a\in A_i^{\#}} C_H(a)$ for each $i$. Since $P$ is abelian, we conclude that $[C_P(B),H]=\prod [C_P(A_i),H]$. Thus, $[C_P(A_i),H]=1$ as $[C_P(A_i),C_H(a)]=1$ for any $a\in  A_i^{\#}$. Hence, $[C_P(B),H]=1$.

The above shows that $C_P(B)\leq Z(G)$ and $C_H(B)$ centralizes $P$. If $H$ is abelian, then $C_G(B)\leq Z(G)$. Hence, $B$ acts fixed-point-freely on $G/Z(G)$ and so $G/Z(G)$ is nilpotent by Thompson's theorem \cite{T}. Consequently, $G$ is nilpotent and so in the case where $P$ and $H$ both are abelian we have $[P,H]=1$. 

Suppose that $H$ is not abelian. By the previous paragraph, $[P,Z(H)]=1$. Considering the action of $H/Z(H)$ on $P$ and arguing by induction on the nilpotency class of $H$ we deduce that $[P,H]=1$. Thus, in the case where $P$ is abelian the lemma holds.

Assume that $P$ is not abelian. Consider the action of $HA$ on $P/\Phi(P)$. By the above, $[P,H]\leq\Phi(P)$. We see that $P=C_P(H)[P,H]\leq C_P(H)\Phi(P)$, which implies that $P=C_P(H)$ and $[P,H]=1$.
\end{proof}

\begin{lemma}\label{Cen1}
Suppose that $P$ is abelian and $P=[P,H]$. Assume that the order of $[C_P(a),C_H(a)]$ is at most $m$ for any $a\in A^{\#}$. Then $\langle C_P(B)^H\rangle$ has $(m,q)$-bounded order.
\end{lemma}
\begin{proof} Recall that by Lemma \ref{R1} we have $C_P(B)=\prod C_P(A_i)$, where $A_1,\dots,A_{q+1}$ are the subgroups of order $q^2$ of $A$ containing $B$. First, we prove that the order of $C_P(B)$ is $(m,q)$-bounded. It suffices to bound the order of $C_P(A_i)$ for each $i$. For each $a\in A_i$ we denote by $P_a$ and $H_a$ the centralizers $C_P(a)$ and $C_H(a)$ respectively. It is clear that $P_a$ is normal in $C_G(a)$. Set $D_a=C_{P_a}(H_a)$ and $D_i=\cap_{a\in A_{i}}D_a$. The index of $D_a$ in $P_a$ is $m$ since $|[P_a,H_a]|=m$ and so the index of $D_i$ in $C_P(A_i)$ is $(m,q)$-bounded. Now, let $x\in D_i$. Taking into account that $H=\prod_{a\in A_i^{\#}}H_a$ we deduce that $[x,H]=1$. Thus, $x=1$ since $P=[P,H]$. We conclude that $D_i$ is trivial for each $i$. Therefore $C_P(A_i)$ has $(m,q)$-bounded order for any $i$ as desired.

Now, let $E_a$ be the centralizer of $[P_a,H_a]$ in $H_a$. Note that $E_a$ has $m$-bounded index in $H_a$ since $H_a/E_a$ embeds in the automorphism group of $[P_a,H_a]$. Moreover, $[P_a,E_a]=1$ because $P_a=C_{P_a}(H_a)[P_a,H_a]$. Set $E_i=\langle E_a \ | \ a\in A_i \rangle$. Note that $E_i$ has $(m,q)$-bounded index in $H$ since $H=\prod_{a\in A_i^{\#}}H_a$. Further, note that $[C_P(A_i),E_i]=1$. It becomes clear that $E=\cap_i E_i$ has $(m,q)$-bounded index in $H$ and $[C_P(B),E]=1$. Therefore, $\langle C_P(B)^H\rangle$ has $(m,q)$-bounded order.

\end{proof}

We use $Z_i(K)$ to denote the $i$th term of the upper central series of $K$.
\begin{lemma}\label{order}
Suppose that the order of $[C_P(a),C_H(a)]$ is at most $m$ for any $a\in A^{\#}$. Then the order of $[P,H]$ is $(m,q)$-bounded.
\end{lemma}
\begin{proof}
We can assume that $P=[P,H]$. Note that if $N$ is a normal subgroup of $P$ such that $[N,H]=1$, then $N\leq Z(P)$. Indeed, in this case we have $[P,N]\leq N$ and so $[P,[N,H]]=1$ and $[H,[P,N]]=1$. Consequently $[N,[P,H]]=1$ by the three subgroup lemma.

For a normal $A$-invariant subgroup $M$ of $P$ and $a\in A^{\#}$ we write $j_a(P/M)$ for the order of $[C_{P/M}(a),C_H(a)]$. We write $j_a(P)$ when $M$ is trivial. Set $k(P)=\sum_{a\in A^{\#}}j_a(P)$. It is clear that $k(P)$ is $(m,q)$-bounded. By induction on $k(P)$ we will prove that the nilpotency class of $P$ is at most $t=2k(P)+1$. If $k(P)=q^3-1$ (the smallest possible value for $k(P)$ - it occurs if and only if $[C_P(a),C_H(a)]=1$ for any $a\in A^{\#}$), then $P$ is trivial by Lemma \ref{commT} since $P=[P,H]$. Further, if $P$ is abelian  there is nothing to prove. Suppose that $P$ is not abelian. Then $[Z_2(P),H]\neq 1$ and so, by Lemma \ref{commT}, $[C_{Z_2(P)}(a),C_H(a)]\neq1$ for some $a\in A^{\#}$. Therefore $k(P/Z_2(P))<k(P)$. By induction, $P/Z_2(P)$ has nilpotency class at most $2(k(P)-1)+1$ and so the nilpotency class of $P$ is at most $2k(P)+1$.

Clearly, $|P|$ is bounded in terms of $|P/P'|$ and the nilpotency class of $P$. Hence, by the previous paragraph in order to prove the lemma it is sufficient to prove that the order of $P/P'$ is $(m,q)$-bounded. In particular, without loss of generality we may assume that $P$ is abelian.

By Lemma \ref{Cen1} the subgroup $\langle C_P(B)^{H}\rangle$ has $(m,q)$-bounded order and since $P$ is abelian we conclude that it is normal in $G$. We can pass to the quotient $G/\langle C_P(B)^H\rangle$ and without loss of generality assume that $C_P(B)=1$.

By Lemma \ref{comm} $[P,C_H(B)]$ has $(m,q)$-bounded order. Hence, it is sufficient to show that $[P,H]=[P,C_H(B)]$. First, suppose that $H$ is abelian. Then  $[P,C_P(B)]$ is normal in $G$ and passing to the quotient we can assume that $[P,C_P(B)]=1$. Thus, $C_G(B)=C_H(B)$ belongs to $Z(G)$ and so $G/Z(G)$ admits a fixed-point-free automorphism of prime order $q$. By Thompson's theorem \cite{T}, $G/Z(G)$ is nilpotent. Therefore, $G$ is nilpotent and $[P,H]=1$, as required.

Suppose that $H$ is not abelian. We have proved that $[P,Z(H)]=[P,C_{Z(H)}(B)]$. The subgroup $[P,Z(H)]$ is normal in $G$. Hence, passing to the quotient $G/[P,Z(H)]$ we can consider the action of $H/Z(H)$ on $P$ and arguing by induction on the nilpotency class of $H$ we deduce that $[P,H]=[P,C_P(B)]$. 
\end{proof}

\begin{lemma}\label{subN}
Let $N$ be a normal $HA$-invariant subgroup of $P$. Assume that $P=[P,H]$ and $|[N,H]|=p^n$. Then $N\leq Z_{2n+1}(P)$.
\end{lemma}
\begin{proof}
If $[N,H]=1$, then $[H,[P,N]]=[P,[N,H]]=1$ and so $N\leq Z(P)$ by the three subgroup lemma. Let $M=N\cap Z_2(P)$ and suppose that $N$ is not in $Z(P)$. Thus, it is clear that $M\not\leq Z(P)$ and $[M,H]\neq 1$. Using induction on $n$ with $G$ and $N$ replaced by $G/M$ and $N/M$ respectively we derive that $N/M\leq Z_{2n-1}(P/M)$, whence $N\leq Z_{2n+1}(P)$.
\end{proof}

Recall that a $p$-group is powerful if $G'\leq G^p$ (for $p$ odd) or $G'\leq G^4$ (for $p=2$). The reader can consult \cite{DSMS} for information on such groups. We will use the fact that the rank of a powerful $p$-group coincides with the minimal number of generators. 

For odd prime the next lemma can be found for example in \cite[Lemma 2.2]{Sh2}.
\begin{lemma}\label{expp}
Let $N$ be a group of rank $r$ and prime exponent $p$ if $p$ is odd or exponent $4$ if $p=2$. Then $|N|\leq p^s$ where $s$ is an $r$-bounded number.
\end{lemma}
\begin{proof}
By Theorem 2.13 of \cite{DSMS} $G$ has a powerful characteristic subgroup $N$ of index at most $p^{\mu(r)}$ where $\mu(r)$ is a number depending only on $r$. Corollary 2.8 in \cite{DSMS} shows that $N$ is a product of at most $r$ cyclic subgroups. Therefore, $N$ is of order at most $p^r$ if $p$ is odd or $4^r$ if $p=2$ and the lemma follows.
\end{proof}

\begin{lemma}\label{rank}
Suppose that ${\bf r}([C_P(a),C_H(a)]) \leq r$ for any $a\in A^{\#}$. Then ${\bf r}([P,H])$ is $(r,q)$-bounded.
\end{lemma}
\begin{proof} Without loss of generality we assume that $P=[P,H]$. Let $M$ be any normal $A$-invariant subgroup of exponent $p$ (or exponent 4 if $p=2$) in $P$. By Lemma \ref{expp}, the order of $[C_M(a),C_H(a)]$ is $r$-bounded for any $a\in A$. Hence, applying Lemma \ref{order} we deduce that $[M,H]$ is of $(r,q)$-bounded order. In other words, there exists an $(r,q)$-bounded number $t$ such that $[M,H]\leq p^t$. Applying this argument to $P/\Phi(P)$ we conclude that $P$ can be generated with at most $t$ elements since the minimal number of generators of $P$ is equal to the rank of $P/\Phi(P)$ by the Burnside Basis Theorem.

Let the symbol ${\bf p}$ denote $p$ if $p$ is odd and 4 if $p=2$. Let $N=\gamma_{2t+1}(P)$ where $t$ is as above. We will show that $N$ is powerful, that is $N'\leq N^{\bf p}$. Pass to the quotient $G/N^{\bf p}$  and assume that $N$ has exponent ${\bf p}$. By the first paragraph  $|[N,H]|\leq p^t$ and so $N\leq Z_{2t+1}(P)$ by Lemma \ref{subN}. Note that $[\gamma_i(P),Z_i(P)]$=1 for any positive integer $i$. Therefore, $N$ is abelian modulo $N^{\bf p}$. This means that $N$ is powerful.

We now wish to show that $N$ can be generated with bounded number of elements. We can pass to the quotient $P/\Phi(N)$ and assume that $N$ is elementary abelian. Let $d$ be the minimal number of generators of $P$. For each $n$ the section $\gamma_n(P)/\gamma_{n+1}(P)$ is generated by $d^n$ elements (see for example \cite[Corollary 2.5.6]{khukhro}). Hence, it suffices to bound the nilpotency class of $P$. Recall that we have already proved that under our assumptions $N\leq Z_{2t+1}(P)$. Therefore $P$ is nilpotent of class at most $4t+2$ and so the minimal number of generators of $N$ is $(r,q)$-bounded. 

Since $N$ is powerful we obtain that ${\bf r}(N)$ is $(r,q)$-bounded. The lemma follows since obviously ${\bf r}(P)\leq {\bf r}(P/N)+{\bf r}(N)$.
\end{proof}

We finish this section with a useful result on coprime action. In the next lemma we will use the fact that if $D$ is any coprime group of automorphisms of a finite simple group, then $D$ is cyclic (see for example \cite{GS}). 

\begin{lemma}\label{prodS}
Let $D$ be a non-cyclic $q$-group of order $q^{2}$ acting on a finite $q'$-group $N=S_1\times\dots\times S_t$ which is a direct product of $t$ nonabelian simple groups. Suppose that ${\bf r}(\gamma_{\infty}(C_N(d)))\leq r$ for any $d\in D^{\#}$. Then $t$ is an $(r,q)$-bounded number and each direct factor $S_i$ has rank at most $r$.

\end{lemma}
\begin{proof}
First, we prove that each direct factor $S_i$ has rank at most $r$. Indeed, if $S_i$ is $D$-invariant, then $S_i$ is contained in $C_G(d)$ for some $d\in D^\#$ and so ${\bf r}(S_i)\leq r$ by the hypotheses. Suppose that $S_i$ is not $D$-invariant. Choose $d\in D^\#$ such that $S_i^d\neq S_i$. Write $S=S_i\times S_i^d\dots\times S_i^{d^{q-1}}$. We see that $C_S(d)$ is exactly the diagonal subgroup of $S$ and so $C_S(d)$ is isomorphic to $S_i$. Thus, we conclude that ${\bf r}(S_i)\leq r$.

Now we prove that $t$ is $(r,q)$-bounded. Write $G=K_1\times\dots\times K_s$ where each $K_i$ is a minimal normal $D$-invariant subgroup. Then each $K_i$ is a product of at most $|D|$ simple factors and so $t\leq|D|s$. Therefore it is sufficient to bound $s$.

Let $S_j$ be a direct factor of $K_i$. If $S_j$ is $D$-invariant, then $S_j$ is contained in $C_G(d)$ for some $d\in D^\#$. If $S_j$ is not $D$-invariant, then we can choose $d\in D^\#$ such that $S_i^d\neq S_i$. Now, it is clear that $C_S(d)$ is exactly the diagonal subgroup of $S_i\times S_i^d\dots\times S_i^{d^{q-1}}$ and so $C_S(d)$ is isomorphic to $S_j$. In other words, for every $i$ there exists $d\in D^\#$ such that $C_{K_i}(d)$ contains a subgroup isomorphic to some $S_j$. Therefore $\gamma_\infty(C_{K_i}(d))$ has even order. Since ${\bf r}(\gamma_\infty(C_{G}(d)))\leq r$, it follows that $\gamma_\infty(C_{K_i}(d))$ can have even order for at most $r$ indexes $i$.
Taking into account that there are only $|D|-1$ nontrivial elements in $D$, we deduce that $s\leq(|D|-1)r$.
\end{proof}

\section{Main results}

We will give a detailed proof only for Theorem \ref{main2}. The proof of Theorem \ref{main1} is easier and can be obtained by just obvious modifications of the proof of Theorem \ref{main2}. The following elementary lemma will be useful (for the proof see for example \cite[Lemma 2.4]{AST}).

\begin{lemma}\label{I3} Let $G$ be a finite group such that $\gamma_{\infty}(G)\leq F(G)$. Let $P$ be a Sylow $p$-subgroup of $\gamma_{\infty} (G)$ and $H$ be a Hall $p'$-subgroup of G. Then $P=[P,H]$.
\end{lemma}

Let $F(G)$ denote the Fitting subgroup of a group $G$. Write $F_{0}(G)=1$ and let $F_{i+1}(G)$ be the inverse image of $F(G/F_{i}(G))$. If $G$ is soluble, the least number $h$ such that $F_{h}(G)=G$ is called the Fitting height $h(G)$ of $G$. Let $B$ be a coprime group of automorphisms of a finite soluble group $G$. It was proved in \cite{Tur} that the Fitting height of $G$ is bounded in terms of $h(C_G(B))$ and the number of prime factors of $|B|$ counting multiplicities.  The nonsoluble length $\lambda(G)$ of a finite group $G$ is defined as the minimum number of nonsoluble factors in a normal series of $G$ all of whose factors are either soluble or (non-empty) direct products of nonabelian simple groups. It was proved in \cite{KS3} (see Corollary 1.2) that the nonsoluble length $\lambda(G)$ of a finite group $G$ does not exceed the maximum Fitting height of soluble subgroups of $G$. If $B$ is a coprime group of automorphisms of a finite group $G$, then the nonsoluble length $\lambda(G)$ of $G$ is bounded in terms of $\lambda (C_G(B))$ and the number of prime factors of $|B|$ counting multiplicities \cite{KS1}. 

Let us now assume the hypothesis of Theorem \ref{main2}. Thus, $A$ is a finite group of prime exponent $q$ and order at least $q^3$ acting on a finite $q'$-group $G$ in such a manner that ${\bf r}(\gamma_{\infty} (C_{G}(a))) \leq r$ for any $a\in A^{\#}$. We wish to show that ${\bf r}(\gamma_{\infty} (G))$ is $(r,q)$-bounded. It is clear that $A$ contains a subgroup of order $q^3$. Thus, replacing if necessary $A$ by such a subgroup we may assume that $A$ has order $q^3$. 

Suppose that $G$ is soluble. In that case $C_G(a)$ has $r$-bounded Fitting height for any $a\in A^{\#}$ (see for example Lemma 1.4 of \cite{KS2}). Hence, $G$ has $(r,q)$-bounded Fitting height and we can use induction on $h(G)$. In the case where $h(G)=2$ the proof is immediate from Lemma \ref{rank}. Indeed, let $P$ be a Sylow $p$-subgroup of $\gamma_{\infty}(G)$ and $H$ a Hall $A$-invariant $p'$-subgroup of $G$. Then by Lemmas \ref{rank} and \ref{I3} the rank of $P=[P,H]$ is $(r,q)$-bounded. Therefore the rank of $\gamma_{\infty} (G)$ is $(r,q)$-bounded. Suppose that the Fitting height of $G$ is $h>2$ and let $N=F_2(G)$ be the second term of the Fitting series of $G$. It is clear that the Fitting height of $G/\gamma_{\infty} (N)$ is $h-1$ and $\gamma_{\infty} (N)\leq \gamma_{\infty}(G)$. Hence, by induction we have that $\gamma_{\infty}(G)/\gamma_{\infty} (N)$ has $(r,q)$-bounded rank. Now, the result follows since ${\bf r}(\gamma_{\infty}(G))\leq {\bf r}( \gamma_{\infty}(G)/\gamma_{\infty} (N))+{\bf r}(\gamma_{\infty}(N))$.

We now drop the assumption that $G$ is soluble. Remark that $\lambda(C_G(a))$ is $r$-bounded for any $a\in A^{\#}$ by \cite{KS3} since its soluble subgroups have $r$-bounded Fitting height. Hence, $\lambda (G)$ is $(r,q)$-bounded and we can use induction on $\lambda(G)$.

First, assume that $G=G'$ and $\lambda(G)=1$. Since $G=G'$, it follows that $G/R(G)$ is a product of nonabelian simple groups where $R(G)$ is the soluble radical of $G$. By the above $\gamma_{\infty }(R(G))$ has $(r,q)$-bounded rank. We can factor out $\gamma_{\infty }(R(G))$ and assume $R(G)$ is nilpotent, that is $R(G)=F(G)$.

We now wish to show that the rank of $[F(G),G]$ is $(r,q)$-bounded. Clearly, it  is sufficient to consider the case when $F(G)=P$ where $P$ is a Sylow $p$-subgroup of $F(G)$. Note that if $s$ is a prime different from $p$ and $H$ is an $A$-invariant Sylow $s$-subgroup of $G$, then ${\bf r}(\gamma_{\infty}(PH))$ is $(r,q)$-bounded because $PH$ is soluble. We will require the following observation about finite simple groups.

\begin{lemma}\label{uu}
Let $K$ be a nonabelian finite simple group and $p$ a prime. There exists a prime $s$ different from $p$ such that $K$ is generated by two Sylow $s$-subgroup.
\end{lemma}
\begin{proof}
If $p\neq2$, then we can use Guralnick's result \cite[Theorem A]{G} that $K$ is generated by an involution and a Sylow 2-subgroup. We therefore can take $s=2$. If $p=2$, we can use King's results \cite{Ki} that $K=\langle i, a \rangle$, where $|i|=2$ and $|a|$ is an odd prime. We have $K=\langle a, a^i \rangle$ since this is an $a$-invariant and $i$-invariant subgroup, which is therefore normal. Hence, $K$ is generated by two elements of odd prime order and the lemma follows.
\end{proof}

By Lemma \ref{prodS} the quotient $G/F(G)$ is a product of a $(r,q)$-bounded number of normal $A$-invariant subgroups $K_1\times \cdots \times K_s$ where $K_i$ is a product of at most $|A|$ nonabelian simple groups. Hence, without loss of generality we can assume that  $G/F(G)$ is a product of isomorphic nonabelian simple groups. In view of Lemma \ref{uu} we deduce that $G/P$ is generated by the image of two Sylow $s$-subgroup $H_1$ and $H_2$ where $s$ is a prime different from $p$. On the other hand, $H_1$ and $H_2$ are conjugate of an $A$-invariant Sylow $s$-subgroup of $G$. It follows that both $[P,H_1]$ and $[P,H_2]$ have $(r,q)$-bounded rank.

Let $H=\langle H_1,H_2 \rangle$. Thus $G=PH$. Since $G=G'$, it is clear that $G=[P,H]H$ and $[P,G]=[P,H]$. By Lemma \ref{C1} we have $[P,H]=[P,H_1][P,H_2]$ and therefore the rank of $[P,H]$ is $(r,q)$-bounded. Passing to the quotient $G/[P,G]$ we can assume that $P=Z(G)$. So we are in the situation where $G/Z(G)$ has $(r,q)$-bounded rank. By a theorem of Lubotzky and Mann \cite{LM} (see also \cite{KuS}) the rank of $G'$ is $(r,q)$-bounded as well. Taking into account that $G=G'$ we conclude that the rank of $G$ is $(r,q)$-bounded.

Let us now deal with the case where $G\neq G'$. Let $G^{(l)}$ be the last term of the derived series of $G$. The argument in the previous paragraph shows that ${\bf r}(G^{(l)})$ is $(r,q)$-bounded. Consequently, ${\bf r}(\gamma_{\infty}(G))$ is $(r,q)$-bounded since $G/G^{(l)}$ is soluble and $\gamma_{\infty}(G^{(l)})\leq\gamma_{\infty}(G)$. This proves the theorem in the particular case where $\lambda(G)\leq1$.

Assume that $\lambda (G)\geq 2$. Let $T$ be a characteristic subgroup of $G$ such that $\lambda(T)=\lambda(G)-1$ and $\lambda(G/T)=1$. By induction, the rank of $\gamma_\infty(T)$ is $(r,q)$-bounded. It is clear that $\lambda(G/\gamma_\infty(T))=1$.  Therefore, the result follows since ${\bf r} (\gamma_{\infty}(G)) \leq {\bf r}(G/\gamma_\infty(T))+{\bf r}(\gamma_\infty(T))$.

\end{document}